\documentclass[draft]{llncs}

\usepackage{url,amsmath,amssymb,bussproofs}

\usepackage{xcolor}

\let\phi\varphi

\newcommand{\N}{\mathsf{N}}
\newcommand{\Dten}{\mathrm{D10}}
\newcommand{\Dth}{\mathrm{D13}}

\newcommand{\Nvalid}[1]{\vDash_{\mathrm{N}}{#1}}
\newcommand{\dom}[0]{\mathop{dom}}

\begin{document}

\title{A curious dialogical logic and \\ its composition problem}
\author{Jesse Alama\inst{1} \and Sara L.\ Uckelman\inst{2}\thanks{Both authors were funded by the FCT/NWO/DFG project ``Dialogical Foundations of Semantics'' (DiFoS) in the ESF EuroCoRes programme LogICCC (FCT LogICCC/0001/2007; LogICCC-FP004; DN 231-80-002; CN 2008/08314/GW).}}
\institute{Center for Artificial Intelligence \\ New University of Lisbon \\
\email{j.alama@fct.unl.pt}
\and Institute for Logic, Language, and Computation \\
Universiteit van Amsterdam \\
\email{S.L.Uckelman@uva.nl}}

\maketitle

\begin{abstract}
Dialogue games are two-player logic games between a Proponent who puts forward a logical formula $\phi$ as valid or true and an Opponent who disputes this.  An advantage of the dialogical approach is that it is a uniform framework from which different logics can be obtained through only small variations of the basic rules.  We introduce the \emph{composition problem} for dialogue games as the problem of resolving, for a set $S$ of rules for dialogue games, whether the set of $S$-dialogically valid formulas is closed under modus ponens.  Solving the composition problem is fundamental for the dialogical approach to logic; despite its simplicity, it often requires an indirect solution with the help of significant logical machinery such as cut-elimination.  We give a set $\mathrm{N}$ of dialogue rules that is quite close to a set of rules known to characterize classical propositional logic, and which is evidently well-justified from the dialogical point of view, but whose set $\N$ of dialogically valid formulas is quite peculiar (and non-trivial).  Its peculiarity notwithstanding, the composition problem for $\N$ can be solved directly.
\end{abstract}

\section{Introduction}
Dialogical logic was developed by Lorenzen in the 1950s and by Lorenzen and Lorenz in the 1970s~\cite{Lor,LL}.  Their basis is a two-player logic game between a Proponent ($P$) who lays down a logical formula $\phi$ and attempts to show, by winning the game, that the formula is valid; the other player, Opponent ($O$), disputes this.  As with other logic games~\cite{sep-logic-games}, less attention is paid to actual plays of dialogue games than to the tree of all possible ways the game could go, given an initial formula $\phi$; of particular interest is the existence of a winning strategy for Proponent, which specifies how Proponent can reply to any move of Opponent in such a way that Proponent can win.

Lorenz claimed that Lorenzen's dialogue games offer a new type of semantics for intuitionistic logic and asserts the equivalence between dialogical validity (defined in terms of winning strategies for the Proponent) and intuitionistic derivability~\cite{lor9,lor10}. Lorenz's proof contained some gaps, and later authors sought to fill these gaps; a complete proof can be found in~\cite{fels}.

Dialogue games are not restricted to intuitionistic logic.  By modifying the rules of the game, they can also provide a semantics for classical logic. The dialogical approach can be adapted equally well to capture validity for other logics, such as paraconsistent, connexive, modal and linear logics~\cite{sep-dialogical-logic,ruck}.  All of these extensions of Lorenzen's and Lorenz's initial formulation of dialogue games are achieved by modifying the rules of the game while maintaining the overall dialogical flavor. In this paper we consider a different route: We keep the particle rules unchanged and consider what happens when we \emph{remove} structural rules, rather than adding or modifying them.

The fact that there is no principled restriction on how the dialogical rules can be modified naturally raises the question of when the set of $S$-valid formulas, for a particular set $S$ of dialogical rules, actually corresponds to a logic.  That is, we are interested in identifying desirable properties of the set of $S$-valid formulas in order to give it some logical sensibility.  One such desirable property is that the set be closed under modus ponens: If $\phi$ and $\phi \rightarrow \psi$ are $S$-dialogically valid, then so should $\psi$ be.  We propose to call the problem of resolving whether a set $S$ of rules for dialogue games satisfies this property the \emph{composition problem} for $S$.

Given the aforementioned correspondences between dialogical validity and validity in various logics, a number of dialogical rule sets $S$ for which positive answers to the composition problem already exist, since the sets of formulas valid in connexive logic, various modal logics, etc., are all closed under modus ponens.  However, these positive solutions to the composition problem use a significant amount of logical machinery, specifically translations of dialogical strategies into derivations in some appropriate cut-free proof theory.  These positive results are, to some extent, unsatisfying because they require that one already have a proof theory for the target logic in question, and that this proof system admits cut elimination; in many cases one or both of these may be lacking.  When possible, we prefer direct solutions to the composition problem that, as far as possible, work solely with dialogues and eschew bringing in outside methods.

The plan of the paper is as follows.  In the next section, we provide an introduction to (propositional) dialogical logic. In~\S\ref{comp} we discuss the composition problem in more detail, relating it to the problem of showing that the set of formulas is a logic, and give a dialogical definition of a new sub-classical propositional logic, $\N$.  In~\S\ref{results}, we prove a number of results leading up to a positive solution to the composition problem for $\N$.  Then, in~\S\ref{properties}, we prove some properties about $\N$ towards locating it within the universe of known propositional logics.  We conclude in \S\ref{conc}.

\section{Dialogical logic}\label{dial}
We largely follow Felscher's approach to dialogical logic~\cite{fels}.  For an overview of dialogical logic, see~\cite{sep-dialogical-logic}.

We work with a propositional language; formulas are built from atoms and $\neg$, $\vee$, $\wedge$, and $\rightarrow$.  In addition to formulas, there are the three so-called \emph{symbolic attack} expressions, $?$, $\wedge_{L}$, and $\wedge_{R}$, which are distinct from all the formulas and connectives.  Together formulas and symbolic attacks are called statements; they are what is asserted in a dialogue game.

The rules governing dialogues are divided into two types.  \emph{Particle} rules say how statements can be attacked and defended depending on their main connective.  \emph{Structural} rules define what sequences of attacks and defenses count as dialogues.  Different logics can be obtained by modifying either set of rules.

\begin{table}[t]
  \centering
  \setlength{\tabcolsep}{5pt} 
  \begin{tabular}{c|c|c}
    \textbf{Assertion} & \textbf{Attack} & \textbf{Response}\\
    \hline
    $\phi \wedge \psi$       & $\wedge_{L}$ & $\phi$\\
                             & $\wedge_{R}$ & $\psi$\\
    $ \phi \vee \psi$        & $?$         & $\phi$ or $\psi$\\
    $ \phi \rightarrow \psi$ & $\phi$      & $\psi$\\
    $ \neg\phi$              & $\phi$      & ---                  
  \end{tabular}
\medskip
  \caption{Particle rules for dialogue games}
  \label{tab:particle-rules}
\end{table}
The standard particle rules are given in Table \ref{tab:particle-rules}. According to the first row, there are two possible attacks against a conjunction: The attacker specifies whether the left or the right conjunct is to be defended, and the defender then continues the game by asserting the specified conjunct.  The second row says that there is one attack against a disjunction; the defender then chooses which disjunct to assert.  The interpretation of the third row is straightforward.  The fourth row says that there is no way to defend against the attack against a negation; the only appropriate ``defense'' against an attack on a negation $\neg\phi$ is to continue the game with the new information $\phi$.

These notions can be made precise as follows (following Felscher).
\begin{definition}
  A signed expression is a pair $\langle A,e \rangle$ where $e$ is a statement and $A$ is either $P$ or $O$.  A signed expression is said to be $P$-signed if its first component is $P$ and $O$-signed if its first component is $O$.
  Let $\delta$ be a sequence (that is, a function whose domain is an ordinal) of signed expressions for some and let $\eta$ be a function for which:
\begin{itemize}
\item $\dom(\eta) = \dom(\delta) \setminus \{0\}$, and
\item for every $n$ in $\dom(\delta)$, the value $\eta(n)$ is a pair $[m,Z]$, where $m$ is a natural number less than $n$ and $Z$ is either ``A'' (attack) or ``D'' (defend).
\end{itemize}
Given such functions $\delta$ and $\eta$, the pair $(\delta,\eta)$ is a \emph{dialogue} if it satisfies the three conditions:
\begin{enumerate}
\item If $n$ is even, then $\delta(n)$ is a $P$-signed expression and if $\delta(n)$ is odd, then $\delta(n)$ is an $O$-signed expression.
\item If $\eta(n) = [m,A]$, then $\delta(m)$ is a non-atomic formula and $\delta(n)$ is an attack upon $\delta(m)$ according to the particle rules.
\item If $\eta(n) = [m,D]$, then $\eta(m) = [k,A]$, and $\delta(n)$ is a defense against the attack $\delta(m)$ according to the particle rules.
\end{enumerate}
If $\delta(0)$ is $\langle P,\phi \rangle$, we say that the dialogue $(\delta,\eta)$ commences with $\phi$.
\end{definition}
These skeletal conditions say only that play alternates between Proponent and Opponent (starting with Proponent at move $0$), and that every move (except the initial assertion $\delta(0)$) is either an attack or a defense against some earlier assertion.

Further constraints on the development of a dialogue are given by the structural rules.  In this paper we keep the particle rules fixed, but we shall consider a few variations of the structural rules.  
\begin{definition}
Given a set $S$ of structural rules, an \emph{$S$-dialogue} for a formula $\phi$ is a dialogue commencing with $\phi$ that adheres to the rules of $S$. Proponent \emph{wins} an $S$-dialogue $(\delta,\eta)$ if there is a $k\in\mathbb{N}$ such that $\dom(\delta) = [0,2k]$ and there is no proper extension of $(\delta,\eta)$, that is, there is no signed expression $\langle A,e \rangle$ and no natural number $n$ such that $\delta$ could be extended to the domain $[0,2k+1]$ with the new value $\langle A,e \rangle$, with $\eta$ likewise extended to have the value $[Z,n]$ at $2k+1$.
\end{definition}
\begin{remark}According to this definition, if the dialogue \emph{can} go on, then neither player is said to win; the game proceeds as long as moves are available.\end{remark}

We can now define the notion of an $S$-winning strategy for Proponent. 
\begin{definition}A branch of a rooted tree is a maximal totally ordered set of nodes that includes the root, where order is understood as the immediate ancestor relation.  The \emph{$S$-dialogue tree $T_{S,\phi}$} for a formula $\phi$ is the rooted tree satisfying the conditions:
\begin{itemize}
\item Every branch of $T_{S,\phi}$ is an $S$-dialogue for $\phi$;
\item Every $S$-dialogue for $\phi$ occurs as a branch of $T_{S,\phi}$.
\end{itemize}
\end{definition}
\begin{remark}$S$-dialogue trees for non-atomic formulas can be quite complex, and indeed it often happens that some branches are infinite.  If all branches are infinite, neither player wins.\end{remark}
\begin{definition}An \emph{$S$-winning strategy} $s$ for $P$ for $\phi$ is a rooted subtree of $T_{S,\phi}$ satisfying:
\begin{enumerate}
\item The root of $s$ is the root of $T_{S,\phi}$;
\item Every branch of $s$ is an $S$-dialogue won by $P$;
\item If $k$ is odd and $a$ is a depth-$k$ node of $s$, then $a$ has exactly one child;
\item If $k$ is even and $a$ is a depth-$k$ node of $s$, then $a$ has the same children as does the image of $a$ in $T_{S,\phi}$.
\end{enumerate}
\end{definition}
\begin{remark}Instead of saying ``winning strategy for $P$'' we simply say ``winning strategy''.\end{remark}

This definition says, in the language of trees, that a winning strategy for $P$ is a kind of function saying how Proponent can win given any move by Opponent.  Condition~(1) simply says that the strategy begins at the beginning.  Condition~(2) says that the nodes of a winning strategy are all moves in a dialogue game and that all ways of playing according to the strategy end with a win for Proponent.  Conditions~(3) and~(4) say that Proponent needs to have a unique response to any move the Opponent could make in any of the dialogues that occur as branches in the strategy.

Dialogue games can be used to capture notions of validity.

\begin{definition}
  For a set $S$ of dialogue rules and a formula $\phi$, the relation     $\vDash_{S} \phi$ means that Proponent has an $S$-winning strategy for $\phi$.  If $\nvDash_{S} \phi$, then we say that $\phi$ is $S$-invalid.
\end{definition}
Note that, like usual proof-theoretic characterizations of validity, dialogue validity is an existential notion, unlike the usual model-theoretic notions of validity, which are universal notions.

We now consider some example rule sets.
\begin{definition}
  The rule set $\mathrm{D}$ is comprised of the following structural rules~\cite[p.~220]{fels}:
  \begin{enumerate}
  \item[$(\Dten)$] $P$ may assert an atomic formula only after it has     been asserted by $O$ before: If $\delta(n)=Pa$ and $a$ is atomic, then there exists $m<n$ such that $\delta(m)=Oa$.
  \item[$(\mathrm{D11})$] If $p$ is an $X$-position, and if at $p-1$     there are several open attacks made by $Y$, then only the     \emph{latest} of them may be answered at $p$: If $n(p)=[n,D]$ and     if $n<j<p$, $j-n=0$, $\eta(j)=[i,A]$, then there exists $q$ such     that $j<q<p$, $\eta(q)=[j,D]$.
  \item[$(\mathrm{D12})$] An attack may be answered at most once: For     every $n$ there exists at most one $p$ such that $\eta(p)=[n,D]$.
  \item[$(\Dth)$] A $P$-assertion may be attacked at most once: If $m$     is even, then there exists at most one $n$ such that     $\eta(n)=[m,A]$.
  \end{enumerate}
\end{definition}

Despite its apparent lack of logical meaning, the rule set $\mathrm{D}$ has the following property:
\begin{theorem}[Felscher]
  A formula $\phi$ is intuitionistically valid iff $\vDash_\mathrm{D} \phi$.
\end{theorem}
The proof goes by converting deductions in an intuitionistic sequent calculus to $\mathrm{D}$-winning strategies (via tableaux), and vice versa.  (The conversions are computable.)


\begin{definition}
  The rule set $\mathrm{D}+\mathrm{E}$ is $\mathrm{D}$ plus the following rule:
  \begin{enumerate}
  \item[$(\mathrm{E})$] $O$ can react only upon the immediately     preceding $P$-statement: If $n$ in def($\delta$) is odd, then     $\eta(n)=[n-1,Z]$, $Z=A$ or $Z=D$.
  \end{enumerate}
\end{definition}
As Felscher notes, $\mathrm{E}$ implies $\Dth$, and, for odd $p$ or $n$, also $\mathrm{D11}$ and $\mathrm{D12}$.  What is surprising is that we have that $\vDash_\mathrm{D} \phi$ iff $\vDash_{\mathrm{D}+\mathrm{E}} \phi$~\cite[p.~221]{fels}. Classical logic corresponds to the rule set $\mathrm{CL}:=\mathrm{D10+D13+E}$, that is, dropping rules $\mathrm{D11}$ and $\mathrm{D12}$ (though, again, the presence of $\mathrm{E}$ ensures that the effect of $\mathrm{D11}$ and $\mathrm{D12}$ partly remains).

\section{The composition problem}\label{comp}

The \emph{composition problem} for a set of dialogue rules $S$ asks whether the set $S$ of formulas $\phi$ for which Proponent has a winning strategy in the $S$-dialogue game commencing with $\phi$ is closed under modus ponens. That is, if Proponent has a winning $S$-strategy for $\phi$ and one for $\phi\rightarrow\psi$, can we prove that Proponent has one for~$\psi$?  The composition problem deals with the composition of the set of formulas which make up a logic.  A positive solution to the composition problem can be given by giving one to a related problem, the \emph{strategy composition problem}, which for a set of dialogue rules $S$ asks, given winning $S$-strategies for Proponent for formulas $\phi$ and one for $\phi\rightarrow\psi$, can we \emph{compose} these strategies into one for $\psi$?  Clearly, a positive answer to this problem will also be a positive answer to the more general problem, but the reverse is not the case: It may be possible that some set $S$ of formulas is closed under modus ponens, but the winning strategies which generate the set are not composable.  That is, a positive answer to the composition problem combined with a negative answer to the strategy composition problem indicates the non-constructivity of the positive answer.

A uniform solution to the composition problem for a wide range of rule sets $S$ seems unrealistic.  Even in specific cases, it is by no means clear how one would go about giving a positive solution (though of course a single counterexample suffices for a negative solution).  The composition problem for a semantics is closely related to cut elimination in proof theory: One often-utilized method for proving a positive solution to what we are calling the composition problem is to give a correspondence between winning strategies and proofs or tableaux in proof-theory known to admit cut elimination~\cite{fels,ferm}.  However, when there is no known proof theory for a logic characterized by some particular dialogical semantics, such a solution may not be available.

The composition problem is closely linked to what is conventionally means for a set of sentences to be a logic:
\begin{definition}Given a language $\mathcal{L}$, a \emph{logic} is a set $\mathsf{L}$ of $\mathcal{L}$-formulas which is closed under modus ponens; that is, if $\phi\in\mathsf{L}$ and $\phi\rightarrow\psi\in\mathsf{L}$, then $\psi\in\mathsf{L}$ as well.\end{definition}
We are not requiring that $\mathsf{L}$ be closed under uniform substitution.  Defining `logic' in this way highlights the importance of the composition problem: Solving the composition problem for a given set $\mathsf{L}$ is a prerequisite for declaring $\mathsf{L}$ a logic.  And this definition of `logic' helps bring to light our fundamental question: When are dialogical ``logics'' really logics?

We are now in a position to define the logic that will concern us for the rest of the paper.

Adding $\mathrm{E}$ to the rule set $\mathrm{D}$ did not change the set of validities: Both $\mathrm{D}$ and $\mathrm{D+E}$ correspond to intuitionistic logic.  A natural question then is whether the same holds for classical logic: That is, whether $\mathrm{D10+D13}=\mathrm{D10+D13+E}$.  The primary result of our paper is to show that this identity does not hold.  Our counterexample is the logic $\N$. 
\begin{definition}
  Let $\mathrm{N}=\Dten+\Dth$.  The logic $\N$ is the set of formulas for which $P$ has a winning $\mathrm{N}$-strategy.\end{definition}
Surprisingly, not only will $\N$ turn out to be radically different from classical logic, we show that it diverges considerably from intuitionistic logic as well.

Not all combinations of structural rules with the standard particle rules result in a logic: There are (trivial) rule sets where a negative answer to the composition problem can easily be given.  For example, let $\mathrm{CL'}$ be $\mathrm{CL}$ with $\Dten$ modified such that Proponent is now allowed to also assert atoms in defense of disjunctions.  Then, $\vDash_\mathrm{CL'}p\vee\neg p$ and $\vDash_\mathrm{CL'}(p\vee\neg p)\rightarrow p$, but $\nvDash_\mathrm{CL'}p$. The set of formulas $\mathsf{CL'}$ for which Proponent has a winning $\mathrm{CL'}$-strategy is therefore \emph{not} a logic. Thus, we must justify our calling $\N$ a logic.  In the next section, we prove that the solution to the composition problem for $\N$ is positive, by giving a positive answer to the strategy composition problem, and hence that it is closed under modus ponens and it deserves to be called a logic.

\section{A positive solution to the composition problem for $\N$}\label{results}
In this section we prove the main result of the paper, namely, a positive solution to the composition problem for $\N$.  We begin with some results concerning properties of winning $\mathrm{N}$-strategies.

\begin{theorem}\label{o-defense-infinite-branch}Every branch for an $\mathrm{N}$-dialogue tree that contains a defensive move by $O$ either terminates at an $O$-move, or is infinite.\end{theorem}
\begin{proof}If a branch of an $\mathrm{N}$-dialogue tree contains such a node   $a$ but does not terminate at an $O$-node, then $P$ has a response to some previous assertion of $O$.  But in this case, $O$ can respond to $P$'s move by repeating the earlier defense of $a$ that occurs in the branch; note that D13 rules out only repeated $O$-attacks, not repeated $O$-defenses.  Thus branches containing a defensive move for $O$ that do not end with an $O$-move are infinite. \qed\end{proof}

\begin{corollary}\label{no-ws-with-o-defending}No $\mathrm{N}$-winning strategy contains a branch where $O$ defends.\end{corollary}
\begin{proof}If $s$ were an $\mathrm{N}$-winning strategy with a branch that contains an defensive $O$-node $a$, then, by Theorem~\ref{o-defense-infinite-branch}, every branch of $s$ containing $a$ either terminates at an $O$-move or is infinite.  But since $s$ is a winning strategy, there can be no branches of $s$ that terminate at an $O$-move, nor can there be any infinite branches. \qed\end{proof}
The corollary implies that when every branch of the $\mathrm{N}$-dialogue tree for a formula $\phi$ contains a defensive move by $O$, then $\phi$ is $\mathrm{N}$-invalid.  The converse, interestingly, fails: In the $\mathrm{N}$-dialogue tree for $(p\rightarrow (\neg q\vee\neg r))\rightarrow((\neg p\rightarrow\neg q)\vee(\neg p\rightarrow\neg r)))$, which is $\mathrm{N}$-invalid, $O$ never defends in any branch.

\begin{lemma}[Weakening]\label{weakening-is-valid} If $\Nvalid\psi$, then $\Nvalid\phi \rightarrow \psi$, for all formulas $\phi$.\end{lemma}
\begin{proof}Let $s_{\psi}$ be an $\mathrm{N}$-winning strategy for $\psi$.  The $\mathrm{N}$-dialogue tree $T_{\phi \rightarrow \psi}$ for $\phi \rightarrow \psi$ begins with $P$'s assertion of $\phi \rightarrow \psi$, followed by $O$'s attacking assertion $\phi$.  These first two nodes of $T_{\phi \rightarrow \psi}$ themselves form a two-element chain, $c$. Carry out the following modification on $s_{\psi}$:
\begin{itemize}
\item The root node $r$ of $s_{\psi}$ is an assertion by $P$ of $\psi$, but it is neither an attack or a defense, and it refers to no prior assertion.  Change $r$ so that it is now an assertion by $P$ of $\psi$, but it is now to be understood as an attack against move 1 (which, in the tree $s_{\phi \rightarrow \psi}$ that we eventually define, will be $O$'s attacking assertion $\phi$ against $P$'s assertion of $\phi \rightarrow \psi$);
\item Every non-root node of $s_{\psi}$ refers to some previous assertion number $k$; change this to $k+2$.
\end{itemize}
Call the result of this modification $s_{\psi}^{\prime}$.  Let $s_{\phi \rightarrow \psi}$ by the result of grafting $s_{\psi}^{\prime}$ to the end of $c$.  Claim: $s_{\phi \rightarrow \psi}$ is an $\mathrm{N}$-winning strategy for $P$ for $\phi \rightarrow \psi$.  That $s_{\phi \rightarrow \psi}$ is a subtree of the full $\mathrm{N}$-dialogue tree $T_{\phi \rightarrow\psi}$ with the same root should be clear (the surgery we carried out on $s$ was intended to ensure that).  The more interesting possibility that needs to be ruled out is that in $T_{\phi \rightarrow \psi}$ Opponent can respond in more ways than were possible in $T_{\psi}$.  But this cannot be: $\Dth$ is still in force, so that $O$ can attack $P$'s assertions at most once.  This implies that $O$'s attack against the initial assertion $\phi \rightarrow \psi$ cannot be repeated, so that any attack by $O$ must be against some assertion by $P$ made at some depth $\geq 2$ in $T_{\phi \rightarrow \psi}$; we need not consider defensive moves by $O$ because of Corollary \ref{no-ws-with-o-defending}.  It remains only to show that every branch of $s_{\phi \rightarrow \psi}$ is finite and terminates with a $P$-move.  But this is so because $s_{\psi}$ has the same property. \qed
\end{proof}

\begin{theorem}[Characterization of implication]\label{char} Every $\mathrm{N}$-valid implication $\phi\rightarrow\psi$ satisfies one of the following three conditions:
\begin{enumerate}
\item $\phi$ is atomic.
\item $\phi$ is negated.
\item $\Nvalid\psi$.
\end{enumerate}
\end{theorem}
\begin{proof}Case (3) is just a restatement of Lemma \ref{weakening-is-valid}.  Suppose now that $\phi$ is not atomic and $\psi$ is not an $\mathrm{N}$-validity.  Proceed by cases:
\begin{itemize}
\item If $\phi$ is an implication $\alpha \rightarrow \beta$, then the $\mathrm{N}$-dialogue tree opens with $O$ attacking the initial statement by asserting $\alpha \rightarrow \beta$.  In any $\mathrm{N}$-winning strategy for $(\alpha \rightarrow \beta) \rightarrow \psi$, Proponent cannot attack $O$'s assertion of $\alpha \rightarrow \beta$, because this leaves open the possibility of a defense by $O$, contradicting Corollary~\ref{no-ws-with-o-defending}.  Thus, any winning strategy $s$ must choose, for $P$'s response to $O$'s initial attack, to defend by asserting the consequent $\psi$ of the entire formula, and no branch of $s$ can attack the antecedent implication $\phi \rightarrow \psi$.  By renumbering the reference labels for nodes of $s$ below the $P$'s assertion of $\psi$ in the obvious way (renumber $k$ to $k - 2$), we obtain a winning strategy for $\psi$, contradicting our assumption.
\item Likewise, $\phi$ cannot be a disjunction, nor could it be a conjunction, for similar reasons: In any $\mathrm{N}$-winning strategy $s$ for $(\alpha \vee \beta) \rightarrow \psi$ (or for $(\alpha \wedge \beta) \rightarrow \psi$), Proponent never attacks $\alpha \vee \beta$ (respectively, $\alpha \wedge \beta)$, so we can recover from $s$ a winning strategy for $\psi$, contradicting our assumption.
\end{itemize}
The only possibility left is that $\phi$ is a negation. \qed
\end{proof}
\begin{remark}To illustrate cases (1) and (2) of this classification of valid implications, consider $p \rightarrow p$ and $\neg p \rightarrow \neg p$.  Illustrating (2), we have the more interesting validities $\neg(\phi\wedge\psi)\rightarrow(\neg\phi\vee\neg\psi)$ and $\neg(\phi\vee\psi)\rightarrow(\neg\phi\wedge\neg\psi)$ (the only directions of De Morgan's laws that are $\mathrm{N}$-valid).\end{remark}

\begin{remark}These conditions are not sufficient: The implicational version of mo\-dus ponens, $p\rightarrow((p\rightarrow q)\rightarrow q)$, has an atomic antecedent, but is (surprisingly) $\mathrm{N}$-invalid.\end{remark}

From the Characterization Theorem, with the help of a few simple lemmas, we can prove a positive solution to the composition problem for $\N$:

\begin{lemma}\label{atom}No atomic formula is $\mathrm{N}$-valid.\end{lemma}
\begin{proof}By Rule~$\Dten$, the set of $\mathrm{N}$-dialogue trees for an atomic formula $p$ is empty. \qed\end{proof}

\begin{corollary}$\N$ is consistent.\end{corollary}
Thus, the composition problem for $\N$ is not trivially solved.

\begin{theorem}\label{double-negation}If $\Nvalid{\neg\phi}$, then $\phi$ is a negation $\neg\psi$ and $\Nvalid{\psi}$.\end{theorem}
\begin{proof}By cases:
  \begin{itemize}
  \item $\phi$ cannot be atomic, since no negated atoms are $\mathrm{N}$-valid, by $\Dten$ and the particle rule for negation.
\item $\phi$ cannot be a disjunction $\alpha \vee \beta$ because, once $O$ attacks the negated disjunction by asserting $\alpha \vee \beta$, the only response for $P$ is to attack the disjunction; $O$ can (indeed, must) defend by selecting either the left or the right disjunct, so by Corollary \ref{no-ws-with-o-defending} no winning strategy exists from this unique initial segment of the $\mathrm{N}$-dialogue tree for $\neg (\alpha \vee \beta)$.
\item Likewise, $\phi$ cannot be an implication or a conjunction.
\end{itemize}
Thus $\phi = \neg\psi$ for some formula $\psi$.  A winning $\mathrm{N}$-strategy $s_\psi$ for $P$ for $\psi$ can be obtained by from a winning strategy $s_{\neg\neg\psi}$ for $\neg\neg\psi$ and noting that, by the particle rule for negation, the winning strategy for $\neg\neg\psi$ begins with a unique initial segment of length two, after which $P$ asserts $\psi$, attacking $O$'s assertion of $\neg\psi$.  Simply remove the root and its unique successor from $s_{\neg\neg\psi}$, declare that $P$'s assertion at the new root is neither an attack nor a response, and is a response to no move of $O$; then renumber the reference labels $k$ on all nodes of $s_{\neg\neg\psi}$ by $k - 2$.  This renumbering is coherent because neither $P$ nor $O$ can attack or respond to moves $0$ and $1$, by the particle rule for negation and $\Dth$, so all reference labels are at least $2$. \qed
\end{proof}

\begin{theorem}[Composition]If $\Nvalid{\phi}$ and $\Nvalid{\phi \rightarrow \psi}$, then $\Nvalid{\psi}$.\end{theorem}
\begin{proof}By the Characterization Theorem~\ref{char}, given that $\Nvalid\phi \rightarrow \psi$, it follows that either
\begin{enumerate}
\item $\phi$ is atomic,
\item $\phi$ is negated, or
\item $\psi$ is $\mathrm{N}$-valid.
\end{enumerate}
Case~(3) is the desired conclusion.  Case~(1) is impossible, in light of the assumption that $\Nvalid{\phi}$, by Lemma~\ref{atom}, so the desired conclusion follows vacuously.  It remains to treat case~(2).   By Theorem~\ref{double-negation}, from $\Nvalid{\phi}$, it follows that $\phi = \neg\neg\chi$ for some formula $\chi$.  The beginning of $T_{\phi \rightarrow \psi}$ can be found at the top of Table~\ref{attack1}.
Since these are the first two steps of an $\mathrm{N}$-winning strategy for $P$, the game does not end here with $O$.  If, in any branch of $s$, Proponent chooses to attack move $1$ by asserting $\neg\chi$ as an attack on $O$'s assertion of $\neg\neg\chi$, then the dialogue would proceed as in Table \ref{attack1}.
\begin{table}[t]
  \centering
  \setlength{\tabcolsep}{5pt} 
  \begin{tabular}{|r|c|l|l|}
    0 & $P$ & $\neg\neg\chi\rightarrow \psi$ & \emph{(initial move)}\\
    1 & $O$ & $\neg\neg\chi$ & [A,0]\\
    \vdots & \vdots & \vdots & \vdots \\
    $k$ & $P$ & $\neg\chi$ & [A,1]\\
    $k+1$ & $O$ & $\chi$ & [A,$k$]
  \end{tabular}
\medskip
\caption{\label{attack1}Branch of $T_{\phi \rightarrow \psi}$ where $P$ attacks $O$'s double negation}
\end{table}
Such a branch ends with $O$, so if there were an $\mathrm{N}$-winning strategy for $P$ that begins in this way, then $P$ must have a response.  Proponent cannot attack $O$'s assertion of $\chi$ at any further point of any branch that begins this way, by Corollary~\ref{no-ws-with-o-defending}.  Thus, $P$ must eventually defend against the attack of move $1$ by asserting $\psi$.  We can conclude that $P$ must actually possess a winning strategy for $\psi$ that can be obtained from $s$ by simply removing all copies of the two-step piece where $P$ attacks $\neg\neg\chi$.  Note that $\chi$ cannot be atomic, by Theorem~\ref{double-negation}, since we are assuming $\Nvalid\neg\neg\chi$.  Thus, deleting all these copies of the two-step exchange cannot affect rule $\Dten$.  Rule $\Dth$ is preserved because if $O$ attacks a $P$-statement in the diminished game then the same $P$-assertion would likewise be attacked multiple times in the original game.\qed
\end{proof}
We have thus shown, via semantic means only, a positive solution to the composition problem for $\N$ and thus we can conclude that it is a logic.  In the next section we move towards characterizing what type of logic $\N$ is.

\section{Properties of $\N$}\label{properties}
Having established that $\N$ is a logic, we next say something about what type of logic it is, and how it fits into the scheme of known propositional logics.
\begin{table}[t]
  \centering
  \begin{tabular}{lll}
    $p\vee\neg p$ & \qquad & $\neg p\vee\neg\neg p$\\
    $(p\rightarrow q)\vee(p\rightarrow\neg q)$ & \qquad & $(p\rightarrow q)\vee     (q\rightarrow p)$\\
$\neg\neg p\rightarrow p$ & \qquad & $p\rightarrow\neg\neg p$\\
$p\rightarrow(p\vee q)$ & \qquad & $p\rightarrow(p\wedge p)$\\
      $\neg p\rightarrow(p\rightarrow q)$ & \qquad & $\neg(p\vee\neg p)\rightarrow q$
\end{tabular}
\medskip
\caption{Some $\mathrm{N}$-validities \label{validities}}
\end{table}
We give some examples of $\mathrm{N}$-valid formulas in Table~\ref{validities}. More generally, we know that 
\begin{theorem}$\N\subset\mathsf{CL}$.\end{theorem}
\begin{proof}Every $\Dten+\Dth$-strategy is also a $\Dten+\Dth+\mathrm{E}$-strategy, by Theorem~\ref{no-ws-with-o-defending}. That the inclusion is strict follows from the fact that $\nvDash_\mathrm{N}(((p\rightarrow q)\rightarrow p)\rightarrow p)$ (Peirce's law), which is classically valid.\qed\end{proof}
As a corollary, $\N$ is not a connexive logic.  We also know that:
\begin{lemma}$\N\nsubseteq\mathsf{IL}$ and $\mathsf{IL}\nsubseteq\N$.\end{lemma}
\begin{proof}For the first claim, $\Nvalid p\vee\neg p$.  For the second claim, $\vDash_\mathsf{IL}(\neg p\vee\neg q)\rightarrow\neg (p\wedge q)$, which, by Theorem \ref{char} is not $\mathrm{N}$-valid, since $\nvDash_\mathrm{N}\neg(p\wedge q)$. \qed \end{proof}
It follows from this that $\N$ is not a relevance logic, since these lie below $\mathsf{IL}$.  Further, since $\N$ is neither sub-intuitionistic nor super-intuitionistic, but is sub-classical, it lies in an interesting and as yet under-investigated part of the lattice of propositional logics.

It turns out that although $\Nvalid\phi\wedge\psi$ iff $\Nvalid\phi$ and $\Nvalid\psi$, and $\Nvalid{\phi \rightarrow \psi}$ implies $\Nvalid{\neg\psi \rightarrow \neg\phi}$, these results and others like them do not hold when formulated as object-language implications.  For example, conjunction elimination ($\phi\wedge\psi\rightarrow\phi$) is not $\mathrm{N}$-valid, and, even more surprisingly, given that $\N$ is closed under modus ponens, neither the conjunctive ($(p\wedge(p\rightarrow q))\rightarrow q$) nor the implicational ($p\rightarrow((p\rightarrow q)\rightarrow q)$) version of modus ponens is $\mathrm{N}$-valid.  Thus, the fact that versions of double negation introduction and elimination are both valid is noteworthy.

$\N$ shares many characteristics with known sub-classical propositional logics, though it does not completely align with any of them.  Like relevance logics, arbitrary uniform substitution is not valid.  Consider, for example, the $\mathrm{N}$-validity $p \rightarrow \neg\neg p$ under the substitution of $p \wedge p$ for $p$: The result of the substitution is $\mathrm{N}$-invalid, because the implication no longer meets any of the requirements in the Characterization Theorem.  Another failure of uniform substitution not immediately given by the Characterization Theorem is the passage from the $\mathrm{N}$-validity $p \vee \neg p$ to $(p \wedge p) \vee \neg(p \wedge p)$.  After $O$'s initial attack, in all branches of the $\mathrm{N}$-dialogue tree $P$ either refrains from asserting $\neg(p \wedge p)$ or asserts it at some move.  Branches where $P$ asserts $\neg(p \wedge p)$ do not lead to a win for $P$ because, after $P$'s assertion of the negation, $O$ must defend by asserting the conjunction.  This leaves $P$ with two options: To attack $O$'s conjunction (and thus fail to win, by Corollary~\ref{no-ws-with-o-defending}), and simply restart the game by defending against the initial attack (so that our analysis of the possible branches recurs, and $P$ does not win).  Branches in which $P$ refrains from asserting $\neg(p \wedge p)$ are infinite because the atomic formula~$p$ is never asserted by $O$, so the only way to play for $P$ is to infinitely repeat the initial defense against the initial attack, which of course does not lead to a win for $P$. The fact that uniform substitution of, e.g., $p\wedge p$ for $p$ in an $\mathrm{N}$-valid formula $\phi$ is not validity preserving points to an curious type of ``resource sensitivity'' in the logic; what is valid with some minimal amount of information may fail to remain valid when more information are provided.  Thus, $\N$ is a type of substructural logic.  Linked to this sensitivity is the fact, illustrated above, that valid inferences cannot be chained together to derive new validities.

\section{Conclusion}\label{conc}

By making a simple and intuitive modification of the usual rules for classical dialogue games, we obtained a set $\N$ of dialogically valid formulas for which we proved a positive answer for its composition problem, thus allowing us to call $\N$ a logic.  Our positive solution to the composition was proved directly through semantic means; we worked solely with dialogue trees and strategies and did not need to follow the usual detour through a cut-free proof system.

The logic $\N$ has a number curious features, including a lack of uniform substitution, and a failure to validate the implicational and conjunctive versions of modus ponens at the object-language level---despite the positive solution to its composition problem---which arise from the fact that if Opponent can defend once, he can always defend.  As a result, this logic privileges implications whose antecedents are atoms or negations, which formulas either cannot be attacked or whose attacks cannot be defended against.  The logic lies below $\mathsf{CL}$, but neither above nor below $\mathsf{IL}$, and is of interest because it is neither connexive nor relevant, two families of well-known non-classical propositional logics which are not superintuitionistic.

\bibliographystyle{splncs03}
\bibliography{icla2011}

\end{document}